\documentclass[preprint,12pt]{elsarticle}
\usepackage{amssymb}
 \usepackage{amsthm}
 \newtheorem{lemma}{Lemma}[section]
 \newtheorem{theorem}{Theorem}[section]
 
\newtheorem{corollary}{Corollary}[section]
\journal{Discrete Appl. Math.}
\begin{document}
\begin{frontmatter}

\title{Equitable chromatic threshold of complete multipartite graphs }
\author{ Zhidan Yan }
\author{ Wei Wang\corref{cor1}}
\cortext[cor1]{Corresponding author.
Fax:+86-997-4682766.\\
\ead{wangwei.math@gmail.com}}

\address{College of Information Engineering, Tarim University, Alar 843300, China}

\begin{abstract}
A proper vertex coloring of a graph is equitable if the sizes of
color classes differ by at most one. The equitable chromatic number
of a graph $G$, denoted by $\chi_=(G)$, is the minimum $k$ such that
$G$ is equitably $k$-colorable. The equitable chromatic threshold of
a graph $G$, denoted by $\chi_=^*(G)$, is the minimum $t$ such that
$G$ is equitably $k$-colorable for $k\ge t$. We develop a formula
and a linear-time algorithm which compute the equitable chromatic
threshold of an arbitrary complete multipartite graph.
\end{abstract}


\end{frontmatter}

\section{Introduction}
\label{intro} All graphs considered in this paper are finite,
undirected and without loops or multiple edges. For a positive
integer $k$, let $[k] = \{1,2,\cdots,k\}$. A proper $k$-coloring of
a graph $G$ is a mapping $f : V(G) \rightarrow [k]$ such that $f(x)
\neq f(y)$ whenever $xy \in E(G)$. We call the set $f^{-1}(i)= \{x
\in V(G) : f(x) = i\}$ a color class for each $i \in [k]$. A graph
is $k$-colorable if it has a $k$-coloring. The chromatic number of
$G$, denoted by $\chi(G)$, is equal to min\{$k$ : $G$ is
$k$-colorable\}. An equitable $k$-coloring of $G$ is a $k$-coloring
for which any two color classes differ in size by at most one, or
equivalently, each color class is of size $\lfloor|V(G)|/k\rfloor$
or $\lceil|V(G)|/k\rceil$. If $G$ has $n$ vertices, we write $n = kq
+ r$ with $0\leq r < k$, then we can rewrite $n = (k-r)q + r(q +
1)$, or equivalently, exactly $r$ (respectively, $k - r$) color
classes have size $q + 1$ (respectively, $q$). The equitable
chromatic number of $G$, denoted by $\chi_ = (G)$, is equal to
min\{$k$ : $G$ is equitably $k$-colorable \}, and the equitable
chromatic threshold of a graph $G$, denoted by $\chi_=^*(G)$, is
equal to min \{$t$ : $G$ is equitably $k$-colorable for $k\geq t$\}.

  The concept of equitable colorability was first introduced by Meyer \cite{W. Meyer1973}. The definitive survey of the subject is
    by Lih \cite{K.-W. Lih1998}.
   Many application such as scheduling and constructing timetables,
    please see \cite{B.Baker1996, S. Irani1996, S. Janson2002, F. Kitagawa1988, M.J. Pelsmajer2004, B.F. Smith1996, A.
    Tucker1973}.

  In 1964, Erd\H{o}s \cite{P. Erdos1964} conjectured that any graph $G$ with maximum degree $\Delta(G)\le k$
   has  an equitable $(k + 1)$-coloring, or equivalently, $\chi_=^*(G)\leq \Delta(G) + 1$. This conjecture was proved
    in 1970 by Hajnal and Szemer\'{e}di \cite{A. Hajnal1970} with a long and complicated proof,
    a polynomial algorithm for such a coloring was found  by Mydlarz and Szemer\'{e}di \cite{M. MydlarzManuscript}.
     Kierstead and Kostochka \cite{H.A. Kierstead2008} gave a short
     proof of the theorem, and presented another polynomial algorithm for such a
     coloring. Brooks' type results are conjectured: Equitable Coloring
     Conjecture \cite{W. Meyer1973} $\chi_=(G) \leq \Delta(G)$, and
     Equitable $\Delta$-Coloring Conjecture \cite{B.-L. Chen1994b} $\chi_=^*(G)\leq
     \Delta(G)$ for $G\notin\{K_n, C_{2n+1}, K_{2n+1,2n+1}\}$. Exact values of
     equitable chromatic numbers of trees \cite{B.-L. Chen1994a} and
     complete multipartite graphs \cite{D.Blum2003},
     \cite{P.C.B. Lam2001} were determined. Our article determines the
     exact value of equitable chromatic threshold of complete multipartite
     graphs.

     The formula which is different from ours was established independently in a manuscript by Chen and Wu, and was reported
     in \cite{K.-W. Lih1998}. However, Chen and Wu never published
     their proof. To our knowledge, this article contains the only published proof.

\section{The results}
 Before stating our main result, we need several  preliminary results on integer partitions.
 Recall that a partition of an integer $n$ is a sum of the form $n = m_1 + m_2 + \cdots + m_k$, where $0\leq m_i \leq n$
  for each $0\leq i \leq k$. We call such a partition a $q$-partition if each $m_i$ is in the set $\{q, q+1\}$.
  A $q$-partition of $n$ is typically denoted as $n = aq + b(q + 1)$, where $n$ is the sum of $a$ $q$'s and $b$ $q + 1$'s.
  A $q$-partition of $n$ is called a minimal $q$-partition if the number of its addends, $a + b$, is as small as possible.
  A $q$-partition of $n$ is called a maximal $q$-partition if the number of its addends, $a + b$, is as large as possible.
  For example, $2 + 2 + 2 + 2$ is a maximal $2$-partition of $8$, and $2 + 3 + 3$ is a minimal $2$-partition of
  $8$. If $q|n$, or equivalently, $n = kq$, with $k \geq 1$, thus we
  write $n = 0(q-1) + kq$ (respectively, $n = kq + 0(q + 1)$), then
  there are both $(q - 1)$-partition and $q$-partition of $n$. For
  example, since $2|8$, we write $8 = 0 \times 1 + 4 \times 2 $ (respectively, $8 = 4 \times 2 + 0 \times 3$), then
  there are both $1$-partition and $2$-partition of $8$.

  Our first lemma is from \cite{D.Blum2003}, which study the condition of which a $q$-partition of
  $n$ exists. For the sake of completeness, here we restate their proof.
   In what follows, all variables are nonegative
  integers.

 \begin{lemma}\label{basic}\cite{D.Blum2003}
 If $0 < q \leq n$, and $n=kq+r$ with $0 \leq r < q$, then there is
 a $q$-partition of $n$ if and only if $r\leq k $.
 \end{lemma}

 \begin{proof}
 If $r\leq k $, then $n = (k - r)q + r(q + 1)$ is a $q$-partition of $n$.
 Conversely, given a $q$-partition $n = aq + b(q + 1)$ of $n$, we have
 $n = (a + b)q + b$, so $(a + b) \leq k$ and $r \leq b$. Consequently, $r \leq b \leq (a + b) \leq
  k$.
 \end{proof}

\begin{corollary}\label{noq-partion}
There is no $q$-partition of $n$ if and only if $n/(q + 1)> \lfloor
n/q \rfloor$.
\end{corollary}

\begin{proof}
Using the division algorithm, write $n = kq + r$, with $0 \leq r <
q$. Then $k = \lfloor n/q \rfloor$, and $r = n - \lfloor n/q \rfloor
q$. Lemma \ref{basic} implies that there is no $q$-partition of $n$
if and only if $r>k$, hence $n - \lfloor n/q \rfloor q
> \lfloor n/q \rfloor$, we can rewrite $n > \lfloor n/q \rfloor (q +
1)$. The Corollary \ref{noq-partion} follows immediately.
\end{proof}

The next two lemmas give conditions under which a $q$-partition of
$n$ is maximal (respectively, minimal).

 \begin{lemma}\label{maximal1}
 A $q$-partition $n = aq + b(q + 1)$ of $n$ is maximal if and only if $b < q$. Moreover a maximal
 $q$-partition is unique.
 \end{lemma}

 \begin{proof}
 Regard $a$ and $b$ as variables, and $q$ as fixed. Solving the
 linear relation $n = aq + b(q + 1)$ for $a$ yields $a + b=(n - b)/q$. Thus
 $a + b$ is a strictly decreasing function of $b$, and moreover $a + b$
 decreases as $b$ increases. Therefore, the $q$-partition
 $n = aq + b(q + 1)$ will be maximal exactly when $b$ is the smallest
 non-negative integer for which $(n - b)/q$ is an integer. Once $b$ is
 fixed, $a$ is determined by the equation $n =aq + b(q + 1)$. Uniqueness
 of maximal $q$-partition follows.

Now suppose $n = aq + b(q + 1)$ is a $q$-partition, and $b < q$. By
what was said in the previous paragraph, $m = (n - b)/q$ is an
integer. If the partition is not maximal, then there are integers
$b^\prime$ and $m^\prime$, with $b > b^\prime \geq 0$ and $m^\prime
> m
> 0$, for which $m^\prime = (n - b^\prime)/q$. Subtracting
$n = m^\prime q + b^\prime$ from $n = mq + b$ gives $b - b^\prime =
(m^\prime-m)q$, so $b >( b - b^\prime) \geq q$.

Conversely, if $n = aq + b(q + 1)$ is a maximal $q$-partition of
$n$, it is impossible for $b \ge q$, for otherwise $n = (a + q + 1)q
+ (b - q)(q + 1)$ is a $q$-partition of $n$ with $a + q + 1 + b
 - q = a + b + 1$ addends, contradicting maximality. Thus, $b < q$.
\end{proof}

\begin{lemma}\label{minimal1}\cite{D.Blum2003}
A $q$-partition $n = aq + b(q + 1)$ of $n$ is minimal if and only if
$a < q + 1$. Moreover a minimal
 $q$-partition is unique.
 \end{lemma}

 Now it is possible to describe exactly the number of addends in a
 maximal (respectively, minimal) $q$-partition.

\begin{lemma}\label{minandmax1}
If $n= aq + b(q + 1)$ is a minimal $q$-partition, then $a + b =
\lceil n/(q + 1)\rceil$. If $n = a^\prime q + b^\prime(q + 1)$ is a
maximal $q$-partition, then $a^\prime + b^\prime = \lfloor n/q
\rfloor$. Moreover, when $\lceil n/(q + 1)\rceil = \lfloor n/q
\rfloor$, there is only one $q$-partition of $n$.
\end{lemma}

\begin{proof}
If $n = aq + b(q + 1)$ is a minimal $q$-partition, then $a + b = (n
+ a)/(q + 1)$, with $a < q + 1$ by Lemma \ref{minimal1}, and so $a +
b = \lceil n/(q + 1)\rceil$. If $n = a^\prime q + b^\prime(q + 1)$
is a maximal $q$-partition, then $a^\prime + b^\prime = (n -
b^\prime)/q$, with $b^\prime < q$ by Lemma\ref{maximal1}, and so
$a^\prime + b^\prime = \lfloor n/q \rfloor$. Now, if $\lceil n/(q +
1)\rceil = \lfloor n/q \rfloor$, then $a + b = a^\prime + b^\prime$.
From Lemma \ref{minimal1} and Lemma \ref{maximal1}, we know that the
minimal (respectively, maximal) $q$-partition is unique.
Consequently, if $\lceil n/(q + 1)\rceil = \lfloor n/q \rfloor$,
then there is only one $q$-partition of $n$.
\end{proof}

\begin{lemma}\label{minandmax2}
Let $n = aq + b(q + 1)$ be the maximal $q$-partition, and $n =
a^\prime (q-1) + b^\prime q$ be the minimal $(q - 1)$-partition. If
$q|n$ then $a + b = a^\prime + b^\prime$, otherwise, $a + b + 1 =
a^\prime + b^\prime$.
\end{lemma}

\begin{proof}
By Lemma \ref{minandmax1}, if $n = aq + b(q + 1)$ is the maximal
$q$-partition, then $a + b = \lfloor n/q \rfloor$. If $n = a^\prime
(q-1) + b^\prime q$ is the minimal $(q - 1)$-partition, Lemma
\ref{minandmax1} implies that $a^\prime + b^\prime = \lceil n/(q - 1
+ 1)\rceil = \lceil n/q\rceil$. Consequently, if $q|n$ then $a + b =
\lfloor n/q \rfloor = \lceil n/q\rceil = a^\prime + b^\prime$,
otherwise, $ a^\prime + b^\prime = \lceil n/q\rceil = \lfloor n/q
\rfloor +1 =a + b + 1$.
\end{proof}

 If $n_1 = a_1q + b_1(q + 1)$,
and $n_2 = a_2q + b_2(q + 1)$ are maximal $q$-partition of $n_1$ and
$n_2$, respectively. If $n_1 = a_1^\prime(q - 1) + b_1^\prime q$,
and $n_2 = a_2^\prime(q - 1) + b_2^\prime q $ are minimal $(q -
1)$-partition of $n_1$ and $n_2$, respectively. Lemma
\ref{minandmax2} implies that
\[
 a_1^\prime + b_1^\prime + a_2^\prime + b_2^\prime  = \left\{ {\begin{array}{c@{{},\quad {}}l}
   a_1 + b_1 + a_2 + b_2 & q|n_1~\mbox{and}~q|n_2\\
 a_1 + b_1 + a_2 + b_2 + 2& q\nmid n_1~\mbox{and}~q\nmid n_2\\
a_1 +b_1 + a_2 + b_2 + 1& (q\nmid n_1~\mbox{and}~q|n_2)~\mbox{or}~
(q|n_1 ~\mbox{and}~ q\nmid n_2).
\end{array}} \right.
\]

These results now combine to give a construction of a minimal
equitable coloring of $K_{n_1,n_2,\cdots,n_l}$ and a method to
change the color classes step by step, so that we can increase the
equitable colors one by one. In words, we must give the computation
of the minimum $t$, when $K_{n_1,n_2,\cdots,n_l}$ can be equitably
$k$-colorable for $k \geq t$.

Denote the partite sets of the graph $K_{n_1,n_2,\cdots,n_l}$ as
$N_1, N_2, \cdots, N_l$, with $|N_i|=n_i$. Any given color class of
an equitable coloring must lie entirely in some $N_i$, for otherwise
two of its vertices are nonadjacent. Thus, any equitable coloring
partitions each $N_i$ into color classes $V_{i_1}, V_{i_2}, \cdots,
V_{i_{v_i}}$, no two of which differ in size by more than one. If
the sizes of the color classes are in the set $\{q,q+1\}$, then
these sizes induce $q$-partitions of each $n_i$. Conversely, given a
number $q$, and $q$-partitions $n_i = a_iq + b_i(q + 1)$, of each
$n_i$, there is an equitable coloring of $K_{n_1,n_2,\cdots,n_l}$
with color sizes $q$ and $q+1$; just partition each $N_i$ into $
a_i$ sets of size $q$, and $ b_i$ sets of $q+1$. It follows, then,
that finding an equitable coloring of $K_{n_1,n_2,\cdots,n_l}$
amounts to finding a number $q$, and simultaneous $q$-partitions of
each of numbers $n_i$. By Corollary \ref{noq-partion}, a necessary
condition for $q$ is that $n_i/(q + 1) \leq \lfloor n_i/q \rfloor$
for all $1 \leq i \leq l$. If we want increase colors one by one,
$q$ must be chosen with the additional property that the total
number of color classes is as small as possible. By Lemmas
 \ref{basic}, \ref{minandmax1} and \ref{minandmax2}, it suffices to choose the minimum $q$ for
which there is $i$ such that $n_i/(q + 1) > \lfloor n_i/q \rfloor$
or there are $n_i$ and $n_j$ , such that $q$ divides neither $n_i$
nor $n_j$ . Equivalently, it suffices to choose the maximum $q - 1$
for which there is $i$ such that $n_i/q \leq \lfloor n_i/(q -
1)\rfloor$, and $(q - 1)|n_j$ , for $j \neq i$. Moreover we can
partition each $n_i$ into $a_i = q \lceil n_i/q \rceil - n_i$ of
sizes $q - 1$, and $b_i = n_i - \lceil ni/q \rceil(q - 1)$ of sizes
$q$.

\begin{theorem}\label{theorem1}
$\chi_=^*(K_{n_1,n_2,\cdots,n_l}) = \sum_{i = 1}^l \lceil n_i/h
\rceil$, where $h$ = min\{$q$ : there is $i$ such that $n_i/(q + 1)
> \lfloor n_i/q \rfloor$ or there are $n_i$ and $n_j$, $i \neq j$, such that $q$
divides neither $n_i$ nor $n_j$\}.
\end{theorem}

\begin{proof}
We prove that $K_{n_1,n_2,\cdots,n_l}$ is equitably $k$-colorable
for any $k \geq \sum_{i=1}^l \lceil n_i/h \rceil$ by induction on
$k$.

First, we prove that $K_{n_1,n_2,\cdots,n_l}$ is equitably $\sum_{i
= 1}^l \lceil n_i/h \rceil$-colorable. Set $h^\prime = h - 1$, by
the definition of $h$, $n_i/(h^\prime + 1) \leq \lfloor n_i/h^\prime
\rfloor$, for $1 \leq i \leq l$. Corollary \ref{noq-partion} implies
that each $n_i$ has an $h^\prime$-partition. Let $n_i=a_ih^\prime +
b_i(h^\prime + 1)$ be the minimal $h^\prime$-partition of each
$n_i$. By Lemma \ref{minandmax1}, $a_i + b_i = \lceil n_i/(h^\prime
+ 1)\rceil = \lceil n_i/h\rceil$, and hence we get an equitable
$\sum_{i = 1}^l \lceil n_i/h \rceil$-coloring of
$K_{n_1,n_2,\cdots,n_l}$. It is straightforward to check that
$K_{n_1,n_2,\cdots,n_l}$ is equitably $\sum_{i = 1}^l \lceil n_i/h
\rceil$-colorable.

Now, we assume that $K_{n_1,n_2,\cdots,n_l}$ is equitably
$k$-colorable for some $k \geq \sum_{i=1}^l \lceil n_i/h \rceil$. It
suffices to prove $K_{n_1,n_2,\cdots,n_l}$ is equitably $(k +
1)$-colorable.

By the assumption, each $n_i$ has a $q$-partition $n_i = a_iq +
b_i(q+1)$ such that $\sum_{i = 1}^l(a_i + b_i) = k $.

\noindent\textbf{Claim 1}   \emph{$0 \leq q \leq h - 1 < h$.}

Suppose to the contrary that $q \geq h$. By Lemma \ref{minandmax1},
$a_i+b_i \leq \lfloor n_i/q\rfloor$, and hence $k = \sum_{i =
1}^l(a_i + b_i)
 \leq \sum_{i = 1}^l\lfloor n_i/q\rfloor \leq \sum_{i = 1}^l n_i/q \leq \sum_{i = 1}^l n_i/h
 $. By the definition of $h$, there are $n_i$ and $n_j$, $i \neq j$, such that $h$
divides neither $n_i$ nor $n_j$, or there is some $n_i$ such that
$n_i/(h + 1)
> \lfloor n_i/h \rfloor$. Either case implies that $h \nmid
 n_i$ for some $n_i$. Hence $k \leq \sum_{i = 1}^l n_i/h < \sum_{i = 1}^l \lceil
n_i/h \rceil$. This is a contradiction to $k \geq \sum_{i=1}^l
\lceil n_i/h \rceil$. The claim follows.

To prove $K_{n_1,n_2,\cdots,n_l}$ is equitably $(k+1)$-colorable, we
consider two cases.

 \textbf{Case 1}: There is some $n_i$ such that
whose $q$-partition $n_i = a_iq + b_i(q+1)$ is not maximal. By Lemma
\ref{maximal1}, $b_i \geq q$, so we can rewrite $n_i = (a_i + q +
1)q + (b_i - q)(q+1)$. Thus there is a $q$-partition of $n_i$ with $
a_i + q + 1 + b_i - q = a_i + 1 + b_i$ addends. Hence, we get an
equitable $(k + 1)$-coloring of $K_{n_1,n_2,\cdots,n_l}$.

\textbf{Case 2}: Each $q$-partition $n_i = a_iq + b_i(q+1)$ is
maximal. By Claim 1, $0 \leq q \leq h - 1 < h$, the definition of
$h$ implies that $q$ divides $n_i$ for all $i$ with at most one
exception.

\textbf{Subcase 2.1}: There is no $i$ such that $q \nmid
  n_i$, in other words, $q|n_i$ for all $i$. By Lemma
 \ref{minandmax2}, each maximal $q$-partition is the minimal
 $(q-1)$-partition of $n_i$. Since $0 \leq q - 1 \leq h - 2 < h$, it implies that $q - 1$ divides
 $n_i$ for all $i$ with at most one
exception. Consequently, there is some $n_j$ such that $(q - 1)|n_j$
and $q|n_j$, and the number of addends of minimal (respectively,
maximal) $(q-1)$-partition is equal to $\lceil n_j/q \rceil = n_j/q$
(respectively, $\lfloor n_j/(q-1) \rfloor = n_j/(q - 1)$) by Lemma
\ref{minandmax1}. Since $n_j/(q-1) > n_j/q$, the minimal
$(q-1)$-partition is not the maximal $(q-1)$-partition of $n_j$.
Thus, the minimal $(q -1)$-partition of $n_j$ is just not maximal.
So it turn into case 1. So we can obtain an equitable $(k +
1)$-coloring of $K_{n_1,n_2,\cdots,n_l}$.

\textbf{Subcase 2.2}: There is exactly an $i$ such that $q \nmid
 n_i$, and at the same time, $q|n_j$ for  $j \neq i$, with $1 \leq j \leq
 l$. By Lemma \ref{minandmax2}, each maximal $q$-partition of $n_j$ is the minimal
 $(q-1)$-partition of $n_j$. Since $q \nmid
 n_i$, and $q < h$, by the definition of $h$, $n_i/q
\leq \lfloor n_i/(q - 1) \rfloor$. Corollary \ref{noq-partion}
implies that $n_i$ has a $(q - 1)$-partition. Let the partition $n_i
= a_iq + b_i(q+1)$ (respectively, the partition $n_i = a_i^\prime (q
- 1) + b_i^\prime q$) be the maximal $q$-partition (respectively,
the minimal $(q - 1)$-partition) of $n_i$, the number of addends
$a_i^\prime + b_i^\prime$ is equal to $a_i + b_i + 1$ by Lemma
\ref{minandmax1}. So we obtain an equitable $k+1$-coloring of
$K_{n_1,n_2,\cdots,n_l}$.

In a word ,we has proved that $\chi_=^*(K_{n_1,n_2,\cdots,n_l}) \leq
\sum_{i = 1}^l \lceil n_i/h \rceil$.

Next we prove that $K_{n_1,n_2,\cdots,n_l}$ is not equitably
($\sum_{i = 1}^l \lceil n_i/h \rceil - 1)$-colorable.

Suppose to the contrary that $K_{n_1,n_2,\cdots,n_l}$ is equitably
($\sum_{i = 1}^l \lceil n_i/h \rceil - 1)$-colorable. Then, each
$n_i$ has a $q$-partition $n_i = a_iq + b_i(q+1)$ such that $k =
\sum_{i = 1}^l(a_i + b_i) = \sum_{i = 1}^l \lceil n_i/h \rceil - 1$.

\noindent\textbf{Claim 2}   \emph{$q = h$}

First, we prove that $q \geq h$. Suppose to the contrary that $q
\leq h-1 < h$. By Lemma \ref{minandmax1}, $(a_i + b_i) \geq \lceil
n_i/(q+1) \rceil$, thus $\sum_{i = 1}^l(a_i + b_i) \geq \sum_{i =
1}^l \lceil n_i/(q+1) \rceil \geq \sum_{i = 1}^l \lceil n_i/h
\rceil$. This is a contradiction to $k  = \sum_{i = 1}^l \lceil
n_i/h \rceil - 1$. Second, we prove that $q \leq h$. Suppose to the
contrary that $q > h$. Lemma \ref{minandmax1} implies that $(a_i +
b_i) \leq \lfloor n_i/q \rfloor < \lfloor n_i/h \rfloor$. By the
definition of $h$, there is some $n_i$ such that $n_i \nmid h$,
clearly, $ \lceil n_i/h\rceil - 1 = \lfloor n_i/h \rfloor$. Thus, $k
< \sum_{i = 1}^l \lceil n_i/h \rceil - 1$. This is a contradiction
to $k = \sum_{i = 1}^l \lceil n_i/h \rceil - 1$. The claim follows.

Now, we consider two cases of $h$.

\textbf{case 1}: $h$ = min\{$q$ : there is $i$ such that $n_i/(q +
1)
> \lfloor n_i/q \rfloor$ \}. By Corollary \ref{noq-partion}, there is no $h$-partition of $n_i$. It is
contrary to that each $n_i$ is partitioned into sets of the sizes
$h$ or $h+1$.

\textbf{case 2}: $h$ = min\{$q$ :
 there are $n_i$ and $n_j$, $i \neq j$, such that $q$
divides neither $n_i$ nor $n_j$\}. Let $n_i = a_i^\prime (h - 1) +
b_i^\prime h$, $n_i = a_j^\prime (h - 1) + b_j^\prime h$ be the
minimal $(h - 1)$-partition of $n_i$ and $n_j$, respectively. Let
$n_i = a_ih + b_i(h + 1)$, $n_j = a_jh + b_j(h + 1)$ be the maximal
$h$-partition of $n_i$ and $n_j$, respectively.
Lemma\ref{minandmax2} implies that $ a_i  +  b_i +  a_j +  b_j =
a_i^\prime  +  b_i^\prime  +  a_j^\prime  +  b_j^\prime  - 2 $. And
hence, $\sum_{i = 1}^l(a_i + b_i) \leq \sum_{i = 1}^l \lceil n_i/h
\rceil - 2$. Consequently ,we can not obtain an equitable ($\sum_{i
= 1}^l \lceil n_i/h \rceil - 1$)-coloring of
$K_{n_1,n_2,\cdots,n_l}$.

Therefore, $\chi_=^*(K_{n_1,n_2,\cdots,n_l}) \geq \sum_{i = 1}^l
\lceil n_i/h \rceil$, and so  $\chi_=^*(K_{n_1,n_2,\cdots,n_l})=
\sum_{i = 1}^l \lceil n_i/h \rceil$.
\end{proof}

Theorem \ref{theorem1} leads immediately to an algorithm which finds
the minimal equitable coloring of $K_{n_1,n_2,\cdots,n_l}$ such that
we can increase the colors one by one, through we adjust the
partition of $n_i$ step by step.

\textbf{Equitable Chromatic threshold algorithm}

Let $K_{n_1,n_2,\cdots,n_l}$ be a complete multipartite graph, where
the partite sets of the graph $K_{n_1,n_2,\cdots,n_l}$ are denoted
as $N_1, N_2, \cdots, N_l$, with $|N_i|=n_i$. Let $s^*$ =
min\{$s_i^*$, where $s_i^*$ is the minimum positive integer such
that $s_i^*\nmid n_i$\}.

(0) Set $h = s^*$.

(1) If there are $n_i$ and $n_j$, such that $h$ divides neither
$n_i$ nor $n_j$, with $i \neq j$, stop. Otherwise, go to (2).

(2) There is $i$ such that $h\nmid n_i$, and $h|n_j$ with $i \neq
j$. If $n_i/(h + 1)
> \lfloor n_i/h \rfloor$ , stop. Otherwise, go to (3).

(3) Let $h = h + 1$, go to (1).

The equitable Chromatic threshold of $K_{n_1,n_2,\cdots,n_l}$ is
$\sum_{i = 1}^l \lceil n_i/h \rceil$. Notice that the complexity of
the algorithm is linear in $|V(K_{n_1,n_2,\cdots,n_l})|$.

According Theorem \ref{theorem1}, we have the following corollary
which is a W.-H. Lin's\cite{W.-H. Lin2010} result.

\begin{corollary}\label{corollary2}\cite{W.-H. Lin2010}
If integers $n \geq 1$ and $r \geq 2$, then
$\chi_=^*(K_{\underbrace{n,n,\cdots,n}_{r}}) = r\lceil n/s^*
\rceil$, where $s^*$ is the minimum positive integer such that $s^*
\nmid n$.
\end{corollary}

\end{document}